\documentclass[12pt]{article}

\usepackage{amsmath,amsthm,amssymb}
\usepackage[top=30truemm,bottom=30truemm,left=25truemm,right=25truemm]{geometry}
\usepackage[dvips]{graphicx,color,psfrag}
\usepackage{bm}

\theoremstyle{plain}
\newtheorem{definition}{Definition}[section]
\newtheorem{thm}[definition]{Theorem}

\newtheorem{lem}[definition]{Lemma}
\newtheorem{cor}[definition]{Corollary}

\newtheorem{rem}[definition]{Remark}

\title{Integer group determinants for ${\rm C}_{4}^{2}$}
\author{Yuka Yamaguchi and Naoya Yamaguchi}
\date{\today}
\begin{document}

\maketitle

\begin{abstract}
We determine all possible values of the integer group determinant of ${\rm C}_{4}^{2}$, 
where ${\rm C}_{4}$ is the cyclic group of order $4$. 
\end{abstract}

\section{Introduction}
Recall that for a finite group $G$, 
assigning a variable $x_{g}$ for each $g \in G$, 
the group determinant of $G$ is defined as $\det{\left( x_{g h^{- 1}} \right)}_{g, h \in G}$. 
The group determinant is called an integer group determinant of $G$ when the all variables $x_{g}$ are integers. 
We denote the set of all possible values of the integer group determinant of $G$ by $S(G)$; 
That is, 
$$
S(G) := \left\{ \det{\left( x_{g h^{- 1}} \right)}_{g, h \in G} \mid x_{g} \in \mathbb{Z} \right\}. 
$$
For every group $G$ of order at most $15$, 
$S(G)$ was determined (see \cite{MR4363104, MR4056860}). 
Let ${\rm C}_{n}$ be the cyclic group of order $n$ and ${\rm D}_{n}$ be the dihedral group of order $n$. 
For the groups of order $16$, 
the complete descriptions of $S(G)$ were obtained for ${\rm D}_{16}$ \cite[Theorem~5.3]{MR3879399}, ${\rm C}_{16}$ \cite{Yamaguchi}, 
${\rm C}_{2}^{4}$ \cite{https://doi.org/10.48550/arxiv.2209.12446} and ${\rm C}_{8} \times {\rm C}_{2}$ \cite[Theorem~1.5]{https://doi.org/10.48550/arxiv.2203.14420}. 
In this paper,  we determine $S \left( {\rm C}_{4}^{2} \right)$. 
\begin{thm}\label{thm:1.1}
Let $P := \left\{ p \mid p \equiv - 3  \: \: ({\rm mod} \: 8) \: \: \text{is a prime number} \right\}$ and 
\begin{align*}
A &:= \left\{ (8 k - 3) (8 l - 3) (8 m - 3) (8 n - 3) \right. \mid \\ 
&\qquad \qquad \left. k \in \mathbb{Z}, \: 8 l - 3, 8 m - 3, 8 n - 3 \in P, \: k + l \not\equiv m + n \: \: ({\rm mod} \: 2) \right\} \\ 
&\subsetneq \left\{ 16 m - 7 \mid m \in \mathbb{Z} \right\}. 
\end{align*}
Then we have 
\begin{align*}
S \left( {\rm C}_{4}^{2} \right) &= \left\{ 16 m + 1, \: 2^{15} p (2 m + 1), \: 2^{16} m \mid m \in \mathbb{Z}, \: p \in P \right\} \cup A. 
\end{align*}
\end{thm}

There are fourteen groups of order $16$ up to isomorphism \cite{MR1510814, MR1505615}, 
and five of them are abelian. 
Our result leaves ${\rm C}_{4} \times {\rm C}_{2}^{2}$ as the only abelian group of order $16$ for which $S(G)$ has not been determined 
(shortly thereafter, the integer group determinants have been determined for ${\rm C}_{4} \times {\rm C}_{2}^{2}$ \cite{https://doi.org/10.48550/arxiv.2211.14761}, ${\rm D}_{8} \times {\rm C}_{2}$, ${\rm Q}_{8} \times {\rm C}_{2}$ \cite[Theorems~3.1 and 4.1]{https://doi.org/10.48550/arxiv.2211.09930}, ${\rm Q}_{16}$ \cite{https://doi.org/10.48550/arxiv.2302.11688} and ${\rm C}_{2}^{2} \rtimes {\rm C}_{4}$ \cite{yamaguchi2023integer}, where ${\rm Q}_{n}$ denotes the generalized quaternion group of order $n$).

\section{Preliminaries}
For any $\overline{r} \in {\rm C}_{n}$ with $r \in \{ 0, 1, \ldots, n - 1 \}$, 
we denote the variable $x_{\overline{r}}$ by $x_{r}$, 
and let $D_{n}(x_{0}, x_{1}, \ldots, x_{n - 1}) := \det{\left( x_{g h^{- 1}} \right)}_{g, h \in {\rm C}_{n}}$. 
For any $(\overline{r}, \overline{s}) \in {\rm C}_{4}^{2}$ with $r, s \in \{ 0, 1, 2, 3 \}$, 
we denote the variable $y_{(\overline{r}, \overline{s})}$ by $y_{j}$, where $j := r + 4 s$, 
and let $D_{4 \times 4}(y_{0}, y_{1}, \ldots, y_{15}) := \det{\left( y_{g h^{- 1}} \right)}_{g, h \in {\rm C}_{4}^{2}}$. 
From the $G = {\rm C}_{4}$ and $H = \{ \overline{0}, \overline{2} \}$ case of \cite[Theorem~1.1]{https://doi.org/10.48550/arxiv.2203.14420}, 
we have the following corollary. 

\begin{cor}\label{cor:2.1}
We have 
\begin{align*}
D_{4}(x_{0}, x_{1}, x_{2}, x_{3}) 
&= D_{2}(x_{0} + x_{2}, x_{1} + x_{3}) D_{2}\left( x_{0} - x_{2}, \sqrt{- 1} (x_{1} - x_{3}) \right) \\ 
&= \left\{ (x_{0} + x_{2})^{2}- (x_{1} + x_{3})^{2} \right\} \left\{ (x_{0} - x_{2})^{2} + (x_{1} - x_{3})^{2} \right\}. 
\end{align*}
\end{cor}

\begin{rem}\label{rem:2.2}
From Corollary~$\ref{cor:2.1}$, we have $D_{4}(x_{0}, x_{1}, x_{2}, x_{3}) = - D_{4}(x_{1}, x_{2}, x_{3}, x_{0})$. 
\end{rem}

From the $H = K = {\rm C}_{4}$ case of \cite[Theorem~1.1]{https://doi.org/10.48550/arxiv.2202.06952}, 
we have the following corollary. 

\begin{cor}\label{cor:2.3}
Let $D := D_{4 \times 4}(y_{0}, y_{1}, \ldots, y_{15})$. Then we have 
\begin{align*}
D = \prod_{k = 0}^{3} D_{4}\left( \sum_{s = 0}^{3} \sqrt{- 1}^{k s} y_{4 s}, \sum_{s = 0}^{3} \sqrt{- 1}^{k s} y_{1 + 4 s}, \sum_{s = 0}^{3} \sqrt{- 1}^{k s} y_{2 + 4 s}, \sum_{s = 0}^{3} \sqrt{- 1}^{k s} y_{3 + 4 s} \right). 
\end{align*}
\end{cor}

Throughout this paper, 
we assume that $a_{0}, a_{1}, \ldots, a_{15} \in \mathbb{Z}$, and let 
\begin{align*}
b_{i} &:= (a_{i} + a_{i + 8}) + (a_{i + 4} + a_{i + 12}), & 0 \leq i \leq 3, \\ 
c_{i} &:= (a_{i} + a_{i + 8}) - (a_{i + 4} + a_{i + 12}), & 0 \leq i \leq 3, \\ 
d_{i} &:= a_{i} - a_{i + 8}, & 0 \leq i \leq 7, \\ 
\alpha_{i} &:= d_{i} + \sqrt{- 1} d_{i + 4}, & 0 \leq i \leq 3. 
\end{align*}

\begin{rem}\label{rem:2.4}
For any $0 \leq i \leq 3$, the following hold: 
\begin{enumerate}
\item[$(1)$] $b_{i} \equiv c_{i} \equiv d_{i} + d_{i + 4} \pmod{2}$; 
\item[$(2)$] $b_{i} + c_{i} \equiv 2 d_{i} \pmod{4}$; 
\item[$(3)$] $b_{i} - c_{i} \equiv 2 d_{i + 4} \pmod{4}$. 
\end{enumerate}
\end{rem}

Let 
\begin{gather*}
\bm{a} := (a_{0}, a_{1}, \ldots, a_{15}), \qquad 
\bm{b} := (b_{0}, b_{1}, b_{2}, b_{3}), \qquad 
\bm{c} := (c_{0}, c_{1}, c_{2}, c_{3}), \\
\beta := (\alpha_{0} + \alpha_{2})^{2} - (\alpha_{1} + \alpha_{3})^{2}, \qquad 
\gamma := (\alpha_{0} - \alpha_{2})^{2} + (\alpha_{1} - \alpha_{3})^{2}. 
\end{gather*}
From Corollaries~$\ref{cor:2.1}$ and $\ref{cor:2.3}$, 
we have the following relation which will be frequently used in this paper: 
$$
D_{4 \times 4}(\bm{a}) = D_{4}(\bm{b}) D_{4}(\bm{c}) \beta \gamma \overline{\beta \gamma} = D_{4}(\bm{b}) D_{4}(\bm{c}) \beta \overline{\beta} \gamma \overline{\gamma}, 
$$
where $\overline{\alpha}$ denotes the complex conjugate of $\alpha \in \mathbb{C}$. 
From 
\begin{align*}
\beta 
&= (\alpha_{0} + \alpha_{2} + \alpha_{1} + \alpha_{3}) (\alpha_{0} + \alpha_{2} - \alpha_{1} - \alpha_{3}) \\ 
&= \left\{ (d_{0} + d_{2} + d_{1} + d_{3}) + \sqrt{- 1} (d_{4} + d_{6} + d_{5} + d_{7}) \right\} \\ 
&\qquad \times \left\{ (d_{0} + d_{2} - d_{1} - d_{3}) + \sqrt{- 1} (d_{4} + d_{6} - d_{5} - d_{7}) \right\}, \\ 
\gamma 
&= \left\{ (\alpha_{0} - \alpha_{2}) + \sqrt{- 1} (\alpha_{1} - \alpha_{3}) \right\} \left\{ (\alpha_{0} - \alpha_{2}) - \sqrt{- 1} (\alpha_{1} - \alpha_{3}) \right\} \\ 
&= \left\{ (d_{0} - d_{2} - d_{5} + d_{7}) + \sqrt{- 1} (d_{4} - d_{6} + d_{1} - d_{3}) \right\} \\ 
&\qquad \times \left\{ (d_{0} - d_{2} + d_{5} - d_{7}) + \sqrt{- 1} (d_{4} - d_{6} - d_{1} + d_{3}) \right\}, 
\end{align*}
we have the following lemma. 

\begin{lem}\label{lem:2.5}
The following hold: 
\begin{align*}
\beta \overline{\beta} 
&= \left\{ (d_{0} + d_{2} + d_{1} + d_{3})^{2} + (d_{4} + d_{6} + d_{5} + d_{7})^{2} \right\} \\ 
&\qquad \times \left\{ (d_{0} + d_{2} - d_{1} - d_{3})^{2} + (d_{4} + d_{6} - d_{5} - d_{7})^{2} \right\}, \\ 
\gamma \overline{\gamma} 
&= \left\{ (d_{0} - d_{2} - d_{5} + d_{7})^{2} + (d_{4} - d_{6} + d_{1} - d_{3})^{2} \right\} \\ 
&\qquad \times \left\{ (d_{0} - d_{2} + d_{5} - d_{7})^{2} + (d_{4} - d_{6} - d_{1} + d_{3})^{2} \right\}. 
\end{align*}
\end{lem}

\begin{rem}\label{rem:2.6}
From Lemma~$\ref{lem:2.5}$, 
each $\beta \overline{\beta}$ and $\gamma \overline{\gamma}$ is invariant under the replacing $(d_{0}, \ldots, d_{7})$ with $(d_{4}, d_{5}, d_{6}, d_{7}, d_{0}, d_{1}, d_{2}, d_{3})$. 
\end{rem}

\begin{lem}\label{lem:2.7}
We have $D_{4 \times 4}(\bm{a}) \equiv D_{4}(\bm{b}) \equiv D_{4}(\bm{c}) \equiv \beta \overline{\beta} \equiv \gamma \overline{\gamma} \pmod{2}$. 
\end{lem}
\begin{proof}
From Corollary~$\ref{cor:2.1}$, it follows that for any $x_{0}, x_{1}, x_{2}, x_{3} \in \mathbb{Z}$, 
$$
D_{4}(x_{0}, x_{1}, x_{2}, x_{3}) \equiv x_{0} + x_{1} + x_{2} + x_{3} \pmod{2}. 
$$
Also, from Lemma~$\ref{lem:2.5}$, we have 
\begin{align*}
\beta \overline{\beta} \equiv d_{0} + d_{1} + \cdots + d_{7} \equiv \gamma \overline{\gamma} \pmod{2}. 
\end{align*}
Therefore, from Remark~$\ref{rem:2.4}$~(1), 
we have $D_{4}(\bm{b}) \equiv D_{4}(\bm{c}) \equiv \beta \overline{\beta} \equiv \gamma \overline{\gamma} \pmod{2}$. 
\end{proof}

The following lemma is immediately obtained from Lemma~$\ref{lem:2.5}$. 

\begin{lem}\label{lem:2.8}
We have 
\begin{align*}
\beta \overline{\beta} &= \left\{ (d_{0} + d_{2})^{2} + (d_{4} + d_{6})^{2} + (d_{1} + d_{3})^{2} + (d_{5} + d_{7})^{2} \right\}^{2} \\ 
&\qquad - 4 \left\{ (d_{0} + d_{2}) (d_{1} + d_{3}) + (d_{4} + d_{6}) (d_{5} + d_{7}) \right\}^{2}, \\ 
\gamma \overline{\gamma} &= \left\{ (d_{0} - d_{2})^{2} + (d_{4} - d_{6})^{2} + (d_{1} - d_{3})^{2} + (d_{5} - d_{7})^{2} \right\}^{2} \\ 
&\qquad - 4 \left\{ (d_{0} - d_{2}) (d_{5} - d_{7}) - (d_{4} - d_{6}) (d_{1} - d_{3}) \right\}^{2}. 
\end{align*}
\end{lem}

By direct calculation, we have the following lemma. 

\begin{lem}\label{lem:2.9}
We have 
\begin{align*}
&\left\{ (d_{0} + d_{2})^{2} + (d_{4} + d_{6})^{2} + (d_{1} + d_{3})^{2} + (d_{5} + d_{7})^{2} \right\}^{2} \\ 
&\qquad - \left\{ (d_{0} - d_{2})^{2} + (d_{4} - d_{6})^{2} + (d_{1} - d_{3})^{2} + (d_{5} - d_{7})^{2} \right\}^{2} \\ 
&\quad= 8 \left( d_{0}^{2} + d_{2}^{2} + d_{4}^{2} + d_{6}^{2} + d_{1}^{2} + d_{3}^{2} + d_{5}^{2} + d_{7}^{2} \right) (d_{0} d_{2} + d_{4} d_{6} + d_{1} d_{3} + d_{5} d_{7}). 
\end{align*}
\end{lem}

\begin{lem}\label{lem:2.10}
The following hold: 
\begin{enumerate}
\item[$(1)$] $2 (d_{0} d_{2} + d_{4} d_{6} + d_{1} d_{3} + d_{5} d_{7}) \equiv b_{0} b_{2} + b_{1} b_{3} + c_{0} c_{2} + c_{1} c_{3} \pmod{4}$; 
\item[$(2)$] $2 (d_{0} d_{7} + d_{2} d_{5} + d_{4} d_{3} + d_{6} d_{1}) \equiv b_{0} b_{3} + b_{2} b_{1} - c_{0} c_{3} - c_{2} c_{1} \pmod{4}$; 
\item[$(3)$] $2 (d_{0} d_{3} + d_{2} d_{1} + d_{4} d_{7} + d_{6} d_{5}) \equiv b_{0} b_{3} + b_{2} b_{1} + c_{0} c_{3} + c_{2} c_{1} \pmod{4}$; 
\item[$(4)$] $2 (d_{0} d_{5} + d_{2} d_{7} + d_{4} d_{1} + d_{6} d_{3}) \equiv b_{0} b_{1} + b_{2} b_{3} - c_{0} c_{1} - c_{2} c_{3} \pmod{4}$; 
\item[$(5)$] $2 (d_{0} d_{1} + d_{2} d_{3} + d_{4} d_{5} + d_{6} d_{7}) \equiv b_{0} b_{1} + b_{2} b_{3} + c_{0} c_{1} + c_{2} c_{3} \pmod{4}$. 
\end{enumerate}
\end{lem}
\begin{proof}
We obtain (1) from 
\begin{align*}
&b_{0} b_{2} + b_{1} b_{3} + c_{0} c_{2} + c_{1} c_{3} \\ 
&\quad= (a_{0} + a_{8} + a_{4} + a_{12}) (a_{2} + a_{10} + a_{6} + a_{14}) + (a_{1} + a_{9} + a_{5} + a_{13}) (a_{3} + a_{11} + a_{7} + a_{15}) \\ 
&\qquad \: \: + (a_{0} + a_{8} - a_{4} - a_{12}) (a_{2} + a_{10} - a_{6} - a_{14}) + (a_{1} + a_{9} - a_{5} - a_{13}) (a_{3} + a_{11} - a_{7} - a_{15}) \\ 
&\quad= 2 (a_{0} + a_{8}) (a_{2} + a_{10}) + 2 (a_{4} + a_{12}) (a_{6} + a_{14}) \\ 
&\qquad \: \: + 2 (a_{1} + a_{9}) (a_{3} + a_{11}) + 2 (a_{5} + a_{13}) (a_{7} + a_{15}) \\ 
&\quad\equiv 2 (d_{0} d_{2} + d_{4} d_{6} + d_{1} d_{3} + d_{5} d_{7}) \pmod{4}. 
\end{align*}
In the same way, we can prove (2)--(5). 
\end{proof}
Let $P := \left\{ p \mid p \equiv - 3  \: \: ({\rm mod} \: 8) \: \: \text{is a prime number} \right\}$. 
It is well known that a positive integer $n$ is expressible as a sum of two squares if and only if in the prime factorization of $n$, 
every prime of the form $4 k + 3$ occurs an even number of times. 
From this, we have the following corollary. 

\begin{cor}\label{cor:2.11}
If $a^{2} + b^{2} \equiv - 3 \pmod{8}$, 
then there exist $k \in \mathbb{Z}$ and $8 l - 3 \in P$ satisfying $a^{2} + b^{2} = (8 k + 1) (8 l - 3)$. 
\end{cor}

\section{Integer group determinant of ${\rm C}_{4}$}

\begin{lem}[{\cite[Lemmas~$4.6$ and $4.7$]{https://doi.org/10.48550/arxiv.2203.14420}}]\label{lem:3.1}
For any $k, l, m, n \in \mathbb{Z}$, the following hold: 
\begin{enumerate}
\item[$(1)$] $D_{4}(2 k + 1, 2 l, 2 m, 2 n) \equiv 8 m + 1 \pmod{16}$; 
\item[$(2)$] $D_{4}(2 k, 2 l + 1, 2 m + 1, 2 n + 1) \equiv 8 (k + l + n) - 3 \pmod{16}$. 
\end{enumerate}
\end{lem}

Let $\mathbb{Z}_{\rm odd}$ be the set of all odd numbers. 

\begin{lem}\label{lem:3.2}
For any $k, l, m, n \in \mathbb{Z}$, the following hold: 
\begin{enumerate}
\item[$(1)$] $D_{4}(2 k, 2 l, 2 m, 2 n) \in \begin{cases} 2^{4} \mathbb{Z}_{\rm odd}, & k + m \not\equiv l + n \pmod{2}, \\ 2^{8} \mathbb{Z}, & k + m \equiv l + n \pmod{2}; \end{cases}$
\item[$(2)$] $D_{4}(2 k + 1, 2 l + 1, 2 m + 1, 2 n + 1) \in \begin{cases} 2^{4} \mathbb{Z}_{\rm odd}, & k + m \not\equiv l + n \pmod{2}, \\ 2^{7} \mathbb{Z}_{\rm odd}, & (k + m) (l + n) \equiv - 1 \pmod{4}, \\ 2^{9} \mathbb{Z}, & \text{otherwise}; \end{cases}$
\item[$(3)$] $D_{4}(2 k, 2 l + 1, 2 m, 2 n + 1)$ \\ 
\quad $\in \begin{cases} 2^{5} \mathbb{Z}_{\rm odd}, & k - m \equiv l - n \equiv 1 \pmod{2}, \\ 2^{6} \mathbb{Z}_{\rm odd}, & k \equiv m \pmod{2} \: \: \text{and} \: \: (2 k + 2 l + 1) (2 m + 2 n + 1) \equiv \pm 3 \pmod{8}, \\ 2^{7} \mathbb{Z}, & \text{otherwise}; \end{cases}$
\item[$(4)$] $D_{4}(2 k, 2 l, 2 m + 1, 2 n + 1) \in \begin{cases} 2^{4} \mathbb{Z}_{\rm odd}, & (2 k + 2 m + 1) (2 l + 2 n + 1) \equiv \pm 3 \pmod{8}, \\ 2^{5} \mathbb{Z}, & (2 k + 2 m + 1) (2 l + 2 n + 1) \equiv \pm 1 \pmod{8}.  \end{cases}$
\end{enumerate}
\end{lem}

To prove Lemma~$\ref{lem:3.2}$, 
we remark the following. 

\begin{rem}[{\cite[Remark~3.5]{https://doi.org/10.48550/arxiv.2209.12446}}]\label{rem:3.3}
For any $k, l, m, n \in \mathbb{Z}$, the following hold: 
\begin{enumerate}
\item[$(1)$] $(2 k + 2 l + 1) (2 m + 2 n + 1) \equiv 1 \pmod{8} \iff k - m \equiv - l + n \pmod{4}$; 
\item[$(2)$] $(2 k + 2 l + 1) (2 m + 2 n + 1) \equiv - 1 \pmod{8} \iff k + m \equiv - l - n - 1 \pmod{4}$;  
\item[$(3)$] $(2 k + 2 l + 1) (2 m + 2 n + 1) \equiv 3 \pmod{8} \iff k + m \equiv 1 - l - n \pmod{4}$; 
\item[$(4)$] $(2 k + 2 l + 1) (2 m + 2 n + 1) \equiv - 3 \pmod{8} \iff k - m \equiv 2 - l + n \pmod{4}$.  
\end{enumerate}
\end{rem}

\begin{proof}[Proof of Lemma~$\ref{lem:3.2}$]
We obtain (1) from 
\begin{align*}
D_{4}(2 k, 2 l, 2 m, 2 n) 
&= \left\{ (2 k + 2 m)^{2} - (2 l + 2 n)^{2} \right\} \left\{ (2 k - 2 m)^{2} + (2 l - 2 n)^{2} \right\} \\ 
&= 2^{4} \left\{ (k + m)^{2} - (l + n)^{2} \right\} \left\{ (k - m)^{2} + (l - n)^{2} \right\}. 
\end{align*}
We obtain (2) from 
\begin{align*}
&D_{4}(2 k + 1, 2 l + 1, 2 m + 1, 2 n + 1) \\ 
&\qquad = \left\{ (2 k + 2 m + 2)^{2} - (2 l + 2 n + 2)^{2} \right\} \left\{ (2 k - 2 m)^{2} + (2 l - 2 n)^{2} \right\} \\ 
&\qquad = 2^{4} \left\{ (k + m + 1)^{2} - (l + n + 1)^{2} \right\} \left\{ (k - m)^{2} + (l - n)^{2} \right\}
\end{align*}
and 
\begin{align*}
(k + m + 1)^{2} - (l + n + 1)^{2} &\in 
\begin{cases}
\mathbb{Z}_{\rm odd}, & k + m \not\equiv l + n \pmod{2}, \\ 
2^{2} \mathbb{Z}_{\rm odd}, & (k + m) (l + n) \equiv - 1 \pmod{4}, \\ 
2^{4} \mathbb{Z}, & (k + m) (l + n) \equiv 1 \pmod{4}, \\ 
2^{3} \mathbb{Z}, & k + m \equiv l + n \equiv 0 \pmod{2}, 
\end{cases} \\ 
(k - m)^{2} + (l - n)^{2} &\in 
\begin{cases}
\mathbb{Z}_{\rm odd}, & k + m \not\equiv l + n \pmod{2}, \\ 
2 \mathbb{Z}_{\rm odd}, & k + m \equiv l + n \equiv 1 \pmod{2}, \\ 
2^{2} \mathbb{Z}, & k + m \equiv l + n \equiv 0 \pmod{2}. 
\end{cases}
\end{align*}
We prove (3). 
From Remark~$\ref{rem:3.3}$, 
we have 
\begin{align*}
\left\{ (k + m)^{2} - (l + n + 1)^{2} \right\} \left\{ (k - m)^{2} + (l - n)^{2} \right\} \in 
\begin{cases}
2 \mathbb{Z}_{\rm odd}, & k - m \equiv l - n \equiv 1 \pmod{2}, \\ 
2^{3} \mathbb{Z}, & k \not\equiv m, \: l \equiv n \pmod{2}, \\ 
2^{3} \mathbb{Z}, & \text{when (I)}, \\ 
2^{2} \mathbb{Z}_{\rm odd}, & \text{when (II)}, 
\end{cases}
\end{align*}
where 
\begin{enumerate}
\item[(I)] $k \equiv m \pmod{2}$ and $(2 k + 2 l + 1) (2 m + 2 n + 1) \equiv \pm 1 \pmod{8}$; 
\item[(II)] $k \equiv m \pmod{2}$ and $(2 k + 2 l + 1) (2 m + 2 n + 1) \equiv \pm 3 \pmod{8}$. 
\end{enumerate}
Therefore, (3) is obtained from 
\begin{align*}
D_{4}(2 k, 2 l + 1, 2 m, 2 n + 1) 
&= \left\{ (2 k + 2 m)^{2} - (2 l + 2 n + 2)^{2} \right\} \left\{ (2 k - 2 m)^{2} + (2 l - 2 n)^{2} \right\} \\ 
&= 2^{4} \left\{ (k + m)^{2} - (l + n + 1)^{2} \right\} \left\{ (k - m)^{2} + (l - n)^{2} \right\}. 
\end{align*}
We prove (4). 
From Remark~$\ref{rem:3.3}$, we have 
\begin{align*}
k + m - l - n \in 
\begin{cases}
2^{2} \mathbb{Z}, & (2 k + 2 m + 1) (2 l + 2 n + 1) \equiv 1 \pmod{8}, \\ 
\mathbb{Z}_{\rm odd}, & (2 k + 2 m + 1) (2 l + 2 n + 1) \equiv - 1 \: \: \text{or} \: \: 3 \pmod{8}, \\ 
2 \mathbb{Z}_{\rm odd}, & (2 k + 2 m + 1) (2 l + 2 n + 1) \equiv - 3 \pmod{8}, 
\end{cases} \\ 
k + m + l + n + 1 \in 
\begin{cases}
\mathbb{Z}_{\rm odd}, & (2 k + 2 m + 1) (2 l + 2 n + 1) \equiv 1 \: \: \text{or} \: \: {- 3} \pmod{8}, \\ 
2^{2} \mathbb{Z}, & (2 k + 2 m + 1) (2 l + 2 n + 1) \equiv - 1 \pmod{8}, \\ 
2 \mathbb{Z}_{\rm odd}, & (2 k + 2 m + 1) (2 l + 2 n + 1) \equiv 3 \pmod{8}. 
\end{cases}
\end{align*}
Therefore, (4) is obtained from 
\begin{align*}
&D_{4}(2 k, 2 l, 2 m + 1, 2 n + 1) \\ 
&\: \: \: = \left\{ (2 k + 2 m + 1)^{2} - (2 l + 2 n + 1)^{2} \right\} \left\{ (2 k - 2 m - 1)^{2} + (2 l - 2 n - 1)^{2} \right\} \\ 
&\: \: \: = 2^{3} (k + m + l + n + 1) (k + m - l - n) \left\{ 2 (k - m) (k - m - 1) + 2 (l - n) (l - n - 1) + 1 \right\}. 
\end{align*}
\end{proof}

\section{Impossible odd numbers}
In this section, 
we consider impossible odd numbers. 
Let 
\begin{align*}
A &:= \left\{ (8 k - 3) (8 l - 3) (8 m - 3) (8 n - 3) \right. \mid \\ 
&\qquad \qquad \left. k \in \mathbb{Z}, \: 8 l - 3, 8 m - 3, 8 n - 3 \in P, \: k + l \not\equiv m + n \: \: ({\rm mod} \: 2) \right\}. 
\end{align*}

\begin{lem}\label{lem:4.1}
We have $S({\rm C}_{4}^{2}) \cap \mathbb{Z}_{\rm odd} \subset \left\{ 16 m + 1 \mid m \in \mathbb{Z} \right\} \cup A$. 
\end{lem}

To prove Lemma~$\ref{lem:4.1}$, 
we use the following five lemmas. 

\begin{lem}\label{lem:4.2}
Let $b_{0} + b_{2} \not\equiv b_{1} + b_{3} \pmod{2}$. 
Then we have the following: 
\begin{enumerate}
\item[$(1)$] If $(b_{0} b_{2} + b_{1} b_{3}, c_{0} c_{2} + c_{1} c_{3}) \equiv (0, 0) \: \: \text{or} \: \: (2, 2) \pmod{4}$, then 
$$
D_{4}(\bm{b}) D_{4}(\bm{c}) \in \left\{ (8 k + 1) (8 l + 1) \mid k, l \in \mathbb{Z}, \: \: k \equiv l \: \: ({\rm mod} \: 2) \right\}; 
$$
\item[$(2)$] If $(b_{0} b_{2} + b_{1} b_{3}, c_{0} c_{2} + c_{1} c_{3}) \equiv (0, 2) \: \: \text{or} \: \: (2, 0) \pmod{4}$, then 
$$
D_{4}(\bm{b}) D_{4}(\bm{c}) \in \left\{ (8 k + 1) (8 l + 1) \mid k, l \in \mathbb{Z}, \: \: k \not\equiv l \: \: ({\rm mod} \: 2) \right\}; 
$$
\item[$(3)$] If $(b_{0} b_{2} + b_{1} b_{3}, c_{0} c_{2} + c_{1} c_{3}) \equiv (1, 1) \: \: \text{or} \: \: (- 1, - 1) \pmod{4}$, then 
$$
D_{4}(\bm{b}) D_{4}(\bm{c}) \in \left\{ (8 k - 3) (8 l - 3) \mid k, l \in \mathbb{Z}, \: \: k \equiv l \: \: ({\rm mod} \: 2) \right\}; 
$$
\item[$(4)$] If $(b_{0} b_{2} + b_{1} b_{3}, c_{0} c_{2} + c_{1} c_{3}) \equiv (1, - 1) \: \: \text{or} \: \: (- 1, 1) \pmod{4}$, then 
$$
D_{4}(\bm{b}) D_{4}(\bm{c}) \in \left\{ (8 k - 3) (8 l - 3) \mid k, l \in \mathbb{Z}, \: \: k \not\equiv l \: \: ({\rm mod} \: 2) \right\}. 
$$
\end{enumerate}
\end{lem}
\begin{proof}
First, we prove (1) and (2). 
If $b_{0} b_{2} + b_{1} b_{3} \equiv 0 \pmod{2}$, 
then exactly three of $b_{0}, b_{1}, b_{2}, b_{3}$ are even. 
On the other hand, for any $k, l, m, n \in \mathbb{Z}$, 
$$
m \equiv 0 \pmod{2} \iff (2 k + 1) (2 m) + (2 l) (2 n) \equiv 0 \pmod{4}. 
$$
Therefore, from Remarks~$\ref{rem:2.2}$ and $\ref{rem:2.4}$~(1) and Lemma~$\ref{lem:3.1}$~(1), 
we have (1) and (2). 
Next, we prove (3) and (4). 
If $b_{0} b_{2} + b_{1} b_{3} \equiv 1 \pmod{2}$, 
then exactly one of $b_{0}, b_{1}, b_{2}, b_{3}$ is even. 
On the other hand, for any $k, l, m, n \in \mathbb{Z}$, 
$$
k + l + n \equiv 0 \pmod{2} \iff (2 k) (2 m + 1) + (2 l + 1) (2 n + 1) \equiv 1 \pmod{4}. 
$$
Therefore, from Remarks~$\ref{rem:2.2}$ and $\ref{rem:2.4}$~(1) and Lemma~$\ref{lem:3.1}$~(2), 
we have (3) and (4). 
\end{proof}

The following lemma is immediately obtained from Lemma~$\ref{lem:4.2}$. 

\begin{lem}\label{lem:4.3}
If $b_{0} + b_{2} \not\equiv b_{1} + b_{3} \pmod{2}$, 
then 
$$
D_{4}(\bm{b}) D_{4}(\bm{c}) \equiv 1 - 4 (b_{0} b_{2} + b_{1} b_{3} + c_{0} c_{2} + c_{1} c_{3}) \pmod{16}. 
$$
\end{lem}

\begin{lem}\label{lem:4.5}
Let $b_{0} + b_{2} \not\equiv b_{1} + b_{3} \pmod{2}$ and $\beta \overline{\beta} \equiv - 3 \pmod{8}$. 
Then the following hold: 
\begin{enumerate}
\item[$(1)$] If $b_{0} b_{2} + b_{1} b_{3} + c_{0} c_{2} + c_{1} c_{3} \equiv 0 \pmod{4}$, then 
$$
D_{4}(\bm{b}) D_{4}(\bm{c}) \in \left\{ (8 k - 3) (8 l - 3) \mid k \in \mathbb{Z}, \: 8 l - 3 \in P, \: k \not\equiv l \: \: ({\rm mod} \: 2) \right\}; 
$$
\item[$(2)$] If $b_{0} b_{2} + b_{1} b_{3} + c_{0} c_{2} + c_{1} c_{3} \equiv 2 \pmod{4}$, then 
$$
D_{4}(\bm{b}) D_{4}(\bm{c}) \in \left\{ (8 k - 3) (8 l - 3) \mid k \in \mathbb{Z}, \: 8 l - 3 \in P, \: k \equiv l \: \: ({\rm mod} \: 2) \right\}. 
$$
\end{enumerate}
\end{lem}
\begin{proof}
From $b_{0} + b_{2} \not\equiv b_{1} + b_{3} \pmod{2}$ and Remark~$\ref{rem:2.4}$ (1), 
$d_{0} + d_{2} + d_{4} + d_{6} \not\equiv d_{1} + d_{3} + d_{5} + d_{7} \pmod{2}$ holds. 
Moreover, from $\beta \overline{\beta} \equiv - 3 \pmod{8}$ and Lemma~$\ref{lem:2.8}$, 
exactly one of $d_{0} + d_{2}$, $d_{1} + d_{3}$, $d_{4} + d_{6}$, $d_{5} + d_{7}$ is even. 
For any $k, l \in \mathbb{Z}$ satisfying $D_{4}(\bm{b}) D_{4}(\bm{c}) = (8 k - 3) (8 l - 3)$, 
it follows from Lemma~$\ref{lem:4.3}$ that $k \not\equiv l \pmod{2}$ if $b_{0} b_{2} + b_{1} b_{3} + c_{0} c_{2} + c_{1} c_{3} \equiv 0 \pmod{4}$; 
$k \equiv l \pmod{2}$ if $b_{0} b_{2} + b_{1} b_{3} + c_{0} c_{2} + c_{1} c_{3} \equiv 2 \pmod{4}$. 
Below, we prove that there exist $k \in \mathbb{Z}$ and $8 l - 3 \in P$ satisfying $D_{4}(\bm{b}) D_{4}(\bm{c}) = (8 k - 3) (8 l - 3)$. 
First, suppose that $b_{0} b_{2} + b_{1} b_{3} \equiv 0 \pmod{2}$. 
Then, exactly three of $b_{0}, b_{1}, b_{2}, b_{3}$ are even. 
From Remark~$\ref{rem:2.2}$, we have 
\begin{align*}
D_{4}(\bm{b}) D_{4}(\bm{c}) 
&= D_{4}(b_{1}, b_{2}, b_{3}, b_{0}) D_{4}(c_{1}, c_{2}, c_{3}, c_{0}) \\ 
&= D_{4}(b_{2}, b_{3}, b_{0}, b_{1}) D_{4}(c_{2}, c_{3}, c_{0}, c_{1}) \\ 
&= D_{4}(b_{3}, b_{0}, b_{1}, b_{2}) D_{4}(c_{3}, c_{0}, c_{1}, c_{2}). 
\end{align*}
Therefore, from Remark~$\ref{rem:2.4}$~(1), 
we may assume without loss of generality that $\bm{b} \equiv \bm{c} \equiv (1, 0, 0, 0) \pmod{2}$. 
From Remark~$\ref{rem:2.4}$, we have 
$$
2 (d_{1} + d_{3}) \equiv b_{1} + b_{3} + c_{1} + c_{3} \equiv b_{1} + b_{3} - c_{1} - c_{3} \equiv 2 (d_{5} + d_{7}) \pmod{4}. 
$$
Thus, $d_{1} + d_{3}$ must be odd since $d_{1} + d_{3} \equiv d_{5} + d_{7} \pmod{2}$. 
Hence, from $b_{1} + b_{3} + c_{1} + c_{3} \equiv 2 (d_{1} + d_{3}) \equiv 2 \pmod{4}$, 
we have $(b_{1} + b_{3}, c_{1} + c_{3}) \equiv (0, 2) \: \: \text{or} \: \: (2, 0) \pmod{4}$. 
We consider the case of $(b_{1} + b_{3}, c_{1} + c_{3}) \equiv (0, 2) \pmod{4}$. 
From $\bm{b} \equiv (1, 0, 0, 0) \pmod{2}$ and Lemma~$\ref{lem:3.1}$~(1), 
there exists $j \in \mathbb{Z}$ satisfying $D_{4}(\bm{b}) = 8 j + 1$. 
On the other hand, from $\bm{c} \equiv (1, 0, 0, 0) \pmod{2}$ and $c_{1} + c_{3} \equiv 2 \pmod{4}$, 
we have $c_{1} - c_{3} \equiv 2 \pmod{4}$. 
Thus, from Corollary~$\ref{cor:2.11}$, 
there exist $l' \in \mathbb{Z}$, $8 l - 3 \in P$ satisfying $(c_{0} - c_{2})^{2} + (c_{1} - c_{3})^{2} = (8 l' + 1) (8 l - 3)$. 
Also, there exists $k' \in \mathbb{Z}$ satisfying $(c_{0} + c_{2})^{2} - (c_{1} + c_{3})^{2} = 8 k' - 3$. 
From Corollary~$\ref{cor:2.1}$, 
we have $D_{4}(\bm{c}) = (8 k' - 3) (8 l' + 1) (8 l - 3)$. 
Therefore, there exists $k \in \mathbb{Z}$ satisfying $D_{4}(\bm{b}) D_{4}(\bm{c}) = (8 k - 3) (8 l - 3)$. 
In the same way, the case of $(b_{1} + b_{3}, c_{1} + c_{3}) \equiv (2, 0) \pmod{4}$ can also be proved. 
Next, suppose that $b_{0} b_{2} + b_{1} b_{3} \equiv 1 \pmod{2}$. 
Then, exactly one of $b_{0}, b_{1}, b_{2}, b_{3}$ is even. 
From Remarks~$\ref{rem:2.2}$ and $\ref{rem:2.4}$~(1), 
we may assume without loss of generality that $\bm{b} \equiv \bm{c} \equiv (0, 1, 1, 1) \pmod{2}$. 
In the same way as in the above, 
we have $(b_{1} + b_{3}, c_{1} + c_{3}) \equiv (0, 2) \: \: \text{or} \: \: (2, 0) \pmod{4}$. 
We consider the case of $(b_{1} + b_{3}, c_{1} + c_{3}) \equiv (0, 2) \pmod{4}$. 
From $\bm{c} \equiv (0, 1, 1, 1) \pmod{2}$ and Lemma~$\ref{lem:3.1}$~(2), 
there exists $k' \in \mathbb{Z}$ satisfying $D_{4}(\bm{c}) = 8 k' - 3$. 
On the other hand, from $\bm{b} \equiv (0, 1, 1, 1) \pmod{2}$ and $b_{1} + b_{3} \equiv 0 \pmod{4}$, 
we have $b_{1} - b_{3} \equiv 2 \pmod{4}$. 
Thus, from Corollary~$\ref{cor:2.11}$, 
there exist $l' \in \mathbb{Z}$, $8 l - 3 \in P$ satisfying $(b_{0} - b_{2})^{2} + (b_{1} - b_{3})^{2} = (8 l' + 1) (8 l - 3)$. 
Also, there exists $j \in \mathbb{Z}$ satisfying $(b_{0} + b_{2})^{2} - (b_{1} + b_{3})^{2} = 8 j + 1$. 
From Corollary~$\ref{cor:2.1}$, 
we have $D_{4}(\bm{b}) = (8 j + 1) (8 l' + 1) (8 l - 3)$. 
Therefore, there exists $k \in \mathbb{Z}$ satisfying $D_{4}(\bm{b}) D_{4}(\bm{c}) = (8 k - 3) (8 l - 3)$. 
In the same way, the case of $(b_{1} + b_{3}, c_{1} + c_{3}) \equiv (2, 0) \pmod{4}$ can also be proved. 
\end{proof}

\begin{lem}\label{lem:4.6}
Let $b_{0} + b_{2} \not\equiv b_{1} + b_{3} \pmod{2}$. 
Then we have 
$$
\beta \overline{\beta} - \gamma \overline{\gamma} \equiv 4 (b_{0} b_{2} + b_{1} b_{3} + c_{0} c_{2} + c_{1} c_{3}) \pmod{16}. 
$$
\end{lem}
\begin{proof}
From Remark~$\ref{rem:2.4}$~(1), we have $
d_{0} + d_{2} + d_{4} + d_{6} \not\equiv d_{1} + d_{3} + d_{5} + d_{7} \pmod{2}$. 
Thus, $d_{0}^{2} + d_{2}^{2} + d_{4}^{2} + d_{6}^{2} + d_{1}^{2} + d_{3}^{2} + d_{5}^{2} + d_{7}^{2} \equiv 1 \pmod{2}$. 
Also, since exactly one or three of $d_{0} + d_{2}$, $d_{4} + d_{6}$, $d_{1} + d_{3}$, $d_{5} + d_{7}$ are even, 
it holds that 
\begin{align*}
&\left\{ (d_{0} + d_{2}) (d_{1} + d_{3}) + (d_{4} + d_{6}) (d_{5} + d_{7}) \right\}^{2} \\ 
&\qquad \equiv \left\{ (d_{0} - d_{2}) (d_{5} - d_{7}) - (d_{4} - d_{6}) (d_{1} - d_{3}) \right\}^{2} \pmod{4}. 
\end{align*}
From the above and Lemmas~$\ref{lem:2.8}$, $\ref{lem:2.9}$ and $\ref{lem:2.10}$~(1), we have 
\begin{align*}
\beta \overline{\beta} - \gamma \overline{\gamma} 
&\equiv 8 (d_{0}^{2} + d_{2}^{2} + d_{4}^{2} + d_{6}^{2} + d_{1}^{2} + d_{3}^{2} + d_{5}^{2} + d_{7}^{2}) (d_{0} d_{2} + d_{4} d_{6} + d_{1} d_{3} + d_{5} d_{7}) \\ 
&\equiv 8 (d_{0} d_{2} + d_{4} d_{6} + d_{1} d_{3} + d_{5} d_{7}) \\ 
&\equiv 4 (b_{0} b_{2} + b_{1} b_{3} + c_{0} c_{2} + c_{1} c_{3}) \pmod{16}. 
\end{align*}
\end{proof}

\begin{lem}\label{lem:4.7}
Let $b_{0} + b_{2} \not\equiv b_{1} + b_{3} \pmod{2}$ and $\beta \overline{\beta} \equiv - 3 \pmod{8}$. 
Then the following hold: 
\begin{enumerate}
\item[$(1)$] If $b_{0} b_{2} + b_{1} b_{3} + c_{0} c_{2} + c_{1} c_{3} \equiv 0 \pmod{4}$, then 
$$
\beta \overline{\beta} \gamma \overline{\gamma} \in \left\{ (8 j + 1) (8 m - 3) (8 n - 3) \mid j \in \mathbb{Z}, \: 8 m - 3, 8 n - 3 \in P, \: j \equiv m + n \: \: ({\rm mod} \: 2) \right\}; 
$$
\item[$(2)$] If $b_{0} b_{2} + b_{1} b_{3} + c_{0} c_{2} + c_{1} c_{3} \equiv 2 \pmod{4}$, then 
$$
\beta \overline{\beta} \gamma \overline{\gamma} \in \left\{ (8 j + 1) (8 m - 3) (8 n - 3) \mid j \in \mathbb{Z}, \: 8 m - 3, 8 n - 3 \in P, \: j \not\equiv m + n \: \: ({\rm mod} \: 2) \right\}. 
$$
\end{enumerate}
\end{lem}
\begin{proof}
From Remark~$\ref{rem:2.4}$~(1) and Lemma~$\ref{lem:4.6}$, 
we have $\gamma \overline{\gamma} \equiv \beta \overline{\beta} \equiv - 3 \pmod{8}$. 
Therefore, from Corollary~$\ref{cor:2.11}$, 
there exist $m', n' \in \mathbb{Z}$, $8 m - 3, 8 n - 3 \in P$ satisfying 
$$
\beta \overline{\beta} = (8 m' + 1) (8 m - 3), \quad \gamma \overline{\gamma} = (8 n' + 1) (8 n - 3). 
$$
Let $j := 8 m' n' + m' + n'$. 
Then, $\beta \overline{\beta} \gamma \overline{\gamma} = (8 j + 1) (8 m - 3) (8 n - 3)$. 
From Lemma~$\ref{lem:4.6}$, 
$$
8 (m + m' + n + n') \equiv \beta \overline{\beta} - \gamma \overline{\gamma} \equiv 4 (b_{0} b_{2} + b_{1} b_{3} + c_{0} c_{2} + c_{1} c_{3}) \pmod{16}. 
$$
Therefore, we have $2 (j + m + n) \equiv 2 (m' + n' + m + n) \equiv b_{0} b_{2} + b_{1} b_{3} + c_{0} c_{2} + c_{1} c_{3} \pmod{4}$. 
\end{proof}

\begin{proof}[Proof of Lemma~$\ref{lem:4.1}$]
Let $D_{4 \times 4}(\bm{a}) = D_{4}(\bm{b}) D_{4}(\bm{c}) \beta \overline{\beta} \gamma \overline{\gamma} \in \mathbb{Z}_{\rm odd}$. 
Then, $b_{0} + b_{2} \not\equiv b_{1} + b_{3} \pmod{2}$ holds from $D_{4}(\bm{b}) \in \mathbb{Z}_{\rm odd}$ and Corollary~$\ref{cor:2.1}$. 
Since $\beta \overline{\beta}$ is an odd number expressible in the form $x^{2} + y^{2}$, 
we have $\beta \overline{\beta} \equiv 1 \pmod{4}$. 
Therefore, from Lemma~$\ref{lem:4.6}$, 
\begin{align*}
\beta \overline{\beta} \gamma \overline{\gamma} 
&\equiv \beta \overline{\beta} \left\{ \beta \overline{\beta} - 4 (b_{0} b_{2} + b_{1} b_{3} + c_{0} c_{2} + c_{1} c_{3}) \right\} \\ 
&\equiv (\beta \overline{\beta})^{2} - 4 (b_{0} b_{2} + b_{1} b_{3} + c_{0} c_{2} + c_{1} c_{3}) \pmod{16}. 
\end{align*}
From this and Lemma~$\ref{lem:4.3}$, we have 
\begin{align*}
D_{4}(\bm{b}) D_{4}(\bm{c}) \beta \overline{\beta} \gamma \overline{\gamma} 
&\equiv \left\{ 1 - 4 (b_{0} b_{2} + b_{1} b_{3} + c_{0} c_{2} + c_{1} c_{3}) \right\} \left\{ (\beta \overline{\beta})^{2} - 4 (b_{0} b_{2} + b_{1} b_{3} + c_{0} c_{2} + c_{1} c_{3}) \right\} \\ 
&\equiv (\beta \overline{\beta})^{2} - 8 (b_{0} b_{2} + b_{1} b_{3} + c_{0} c_{2} + c_{1} c_{3}) \pmod{16}. 
\end{align*}
Moreover, from Remark~$\ref{rem:2.4}$~(1), 
it follows that $D_{4}(\bm{b}) D_{4}(\bm{c}) \beta \overline{\beta} \gamma \overline{\gamma} \equiv (\beta \overline{\beta})^{2} \pmod{16}$. 
Therefore, if $\beta \overline{\beta} \equiv 1 \pmod{8}$, 
then $D_{4 \times 4}(\bm{a}) \in \left\{ 16 m + 1 \mid m \in \mathbb{Z} \right\}$. 
If $\beta \overline{\beta} \equiv - 3 \pmod{8}$, 
then we have $D_{4 \times 4}(\bm{a}) \in A'$ from Lemmas~$\ref{lem:4.5}$ and $\ref{lem:4.7}$, 
where 
\begin{align*}
A' &:= \left\{ (8 j + 1) (8 k - 3) (8 l - 3) (8 m - 3) (8 n - 3) \right. \mid \\ 
&\qquad \qquad \left. j, k \in \mathbb{Z}, \: 8 l - 3, 8 m - 3, 8 n - 3 \in P, \: j \not\equiv k + l + m + n \: \: ({\rm mod} \: 2) \right\}. 
\end{align*}
Since $A' = A$, 
the lemma is proved. 
\end{proof}

\section{Impossible even numbers}\label{Section5}

In this section, 
we consider impossible even numbers. 

\begin{lem}\label{lem:5.1}
The following hold: 
\begin{enumerate}
\item[$(1)$] $S({\rm C}_{4}^{2}) \cap 2 \mathbb{Z} \subset 2^{15} \mathbb{Z}$; 
\item[$(2)$] $S({\rm C}_{4}^{2}) \cap 2^{15} \mathbb{Z}_{\rm odd} \subset \left\{ 2^{15} p (2 m + 1) \mid p \in P, \: m \in \mathbb{Z} \right\}$. 
\end{enumerate}
\end{lem}

To prove Lemma~$\ref{lem:5.1}$, 
we use the following five lemmas.

\begin{lem}\label{lem:5.2}
The following hold: 
\begin{enumerate}
\item[$(1)$] If $b_{0} \equiv b_{1} \equiv b_{2} \equiv b_{3} \equiv 0 \pmod{2}$, then 
\begin{align*}
D_{4}(\bm{b}) D_{4}(\bm{c}) \in 
\begin{cases}
2^{8} \mathbb{Z}_{\rm odd}, & b_{0} + b_{1} + b_{2} + b_{3} \equiv c_{0} + c_{1} + c_{2} + c_{3} \equiv 2 \pmod{4}, \\ 
2^{12} \mathbb{Z}, & \text{otherwise}; 
\end{cases}
\end{align*}
\item[$(2)$] If $b_{0} \equiv b_{1} \equiv b_{2} \equiv b_{3} \equiv 1 \pmod{2}$, then 
\begin{align*}
D_{4}(\bm{b}) D_{4}(\bm{c}) \in 
\begin{cases}
2^{8} \mathbb{Z}_{\rm odd}, & b_{0} + b_{1} + b_{2} + b_{3} \equiv c_{0} + c_{1} + c_{2} + c_{3} \equiv 2 \pmod{4}, \\ 
2^{11} \mathbb{Z}, & \text{otherwise}; 
\end{cases}
\end{align*}
\item[$(3)$] If $b_{0} \equiv b_{2} \not\equiv b_{1} \equiv b_{3} \pmod{2}$, then 
\begin{align*}
D_{4}(\bm{b}) D_{4}(\bm{c}) \in 
\begin{cases}
2^{10} \mathbb{Z}_{\rm odd}, & b_{0} - b_{2} \equiv b_{1} - b_{3} \equiv c_{0} - c_{2} \equiv c_{1} - c_{3} \equiv 2 \pmod{4}, \\ 
2^{11} \mathbb{Z}, & \text{otherwise}; 
\end{cases}
\end{align*}
\item[$(4)$] If $b_{0} + b_{2} \equiv b_{1} + b_{3} \equiv 1 \pmod{2}$, then 
\begin{align*}
D_{4}(\bm{b}) D_{4}(\bm{c}) \in 
\begin{cases}
2^{8} \mathbb{Z}_{\rm odd}, & (b_{0} + b_{2}) (b_{1} + b_{3}) \equiv \pm 3, \: \: (c_{0} + c_{2}) (c_{1} + c_{3}) \equiv \pm 3 \pmod{8}, \\ 
2^{9} \mathbb{Z}, & \text{otherwise}. 
\end{cases}
\end{align*}
\end{enumerate}
\end{lem}
\begin{proof}
For any $k, l, m, n \in \mathbb{Z}$, 
\begin{align*}
k + m \not\equiv l + n \pmod{2} \iff 2 k + 2 l + 2 m + 2 n \equiv 2 \pmod{4}, \\ 
k - m \equiv l - n \equiv 1 \pmod{2} \iff 2 k - 2 m \equiv 2 l - 2 n \equiv 2 \pmod{4}. 
\end{align*}
Therefore, from Remarks~$\ref{rem:2.2}$ and $\ref{rem:2.4}$~(1) and Lemma~$\ref{lem:3.2}$, 
the lemma is proved. 
\end{proof}

\begin{lem}\label{lem:5.3}
The following hold: 
\begin{enumerate}
\item[$(1)$] If $b_{0} + b_{2} \equiv b_{1} + b_{3} \equiv 0 \pmod{2}$, then 
\begin{align*}
\beta \overline{\beta} \gamma \overline{\gamma} \in 
\begin{cases}
2^{4} \mathbb{Z}_{\rm odd}, & b_{0} + b_{1} + b_{2} + b_{3} \not\equiv c_{0} + c_{1} + c_{2} + c_{3} \pmod{4}, \\ 
2^{8} \mathbb{Z}, & b_{0} + b_{1} + b_{2} + b_{3} \equiv c_{0} + c_{1} + c_{2} + c_{3} \pmod{4}; 
\end{cases}
\end{align*}
\item[$(2)$] If $b_{0} + b_{2} \equiv b_{1} + b_{3} \equiv 1 \pmod{2}$, then 
\begin{align*}
\beta \overline{\beta} \gamma \overline{\gamma} \in 
\begin{cases}
2^{7} \mathbb{Z}_{\rm odd}, & d \equiv 2 \pmod{4}, \\ 
2^{8} \mathbb{Z}, & d \equiv 0 \pmod{4}, 
\end{cases}
\end{align*}
\end{enumerate}
where 
$$
d := \left\{ (d_{0} + d_{2}) (d_{5} + d_{7}) + (d_{4} + d_{6}) (d_{1} + d_{3}) \right\} \left\{ (d_{0} - d_{2}) (d_{1} - d_{3}) + (d_{4} - d_{6}) (d_{5} - d_{7}) \right\}. 
$$
\end{lem}
\begin{proof}
We prove (1). 
Let $b_{0} + b_{2} \equiv b_{1} + b_{3} \equiv 0 \pmod{2}$. 
Then, from Remark~$\ref{rem:2.4}$~(1), we have $d_{0} + d_{2} + d_{4} + d_{6} \equiv d_{1} + d_{3} + d_{5} + d_{7} \equiv 0 \pmod{2}$. 
Also, from Remark~$\ref{rem:2.4}$~(1) and (2), 
\begin{align*}
&b_{0} + b_{1} + b_{2} + b_{3} \equiv c_{0} + c_{1} + c_{2} + c_{3} \pmod{4} \\ 
&\qquad \iff b_{0} + b_{2} + c_{0} + c_{2} \equiv b_{1} + b_{3} + c_{1} + c_{3} \pmod{4} \\ 
&\qquad \iff d_{0} + d_{2} \equiv d_{1} + d_{3} \pmod{2}. 
\end{align*}
Therefore, if $b_{0} + b_{1} + b_{2} + b_{3} \not\equiv c_{0} + c_{1} + c_{2} + c_{3} \pmod{4}$, 
then $d_{0} + d_{2} \equiv d_{4} + d_{6} \not\equiv d_{1} + d_{3} \equiv d_{5} + d_{7} \pmod{2}$. 
Thus, from Lemma~$\ref{lem:2.5}$, we have $\beta \overline{\beta} \gamma \overline{\gamma} \in 2^{4} \mathbb{Z}_{\rm odd}$. 
If $b_{0} + b_{1} + b_{2} + b_{3} \equiv c_{0} + c_{1} + c_{2} + c_{3} \pmod{4}$, 
then $d_{0} + d_{2} \equiv d_{4} + d_{6} \equiv d_{1} + d_{3} \equiv d_{5} + d_{7} \pmod{2}$. 
Thus, from Lemma~$\ref{lem:2.5}$, we have $\beta \overline{\beta} \gamma \overline{\gamma} \in 2^{8} \mathbb{Z}$. 
We prove (2). 
Let $b_{0} + b_{2} \equiv b_{1} + b_{3} \equiv 1 \pmod{2}$. 
Then, from Remark~$\ref{rem:2.4}$~(1), we have $d_{0} + d_{2} + d_{4} + d_{6} \equiv d_{1} + d_{3} + d_{5} + d_{7} \equiv 1 \pmod{2}$. 
From Remark~$\ref{rem:2.6}$, we may assume without loss of generality that $d_{0} + d_{2} \equiv 0$, $d_{4} + d_{6} \equiv 1 \pmod{2}$. 
We divide the proof into the following two cases: 
\begin{enumerate}
\item[(i)] $d_{0} + d_{2} \equiv d_{1} + d_{3} \equiv 0$, $d_{4} + d_{6} \equiv d_{5} + d_{7} \equiv 1 \pmod{2}$; 
\item[(ii)] $d_{0} + d_{2} \equiv d_{5} + d_{7} \equiv 0$, $d_{4} + d_{6} \equiv d_{1} + d_{3} \equiv 1 \pmod{2}$. 
\end{enumerate}
We remark that if (i), then $d \equiv d_{0} + d_{2} + d_{1} + d_{3} \pmod{4}$ holds, 
and if (ii), then $d \equiv d_{0} - d_{2} + d_{5} - d_{7} \pmod{4}$ holds. 
First, suppose that (i) holds. 
Then, from Lemma~$\ref{lem:2.5}$, we have $\gamma \overline{\gamma} \in 2^{2} \mathbb{Z}_{\rm odd}$. 
Also, from 
\begin{align*}
(d_{0} + d_{2} + d_{1} + d_{3}, d_{0} + d_{2} - d_{1} - d_{3}) &\equiv (0, 0) \: \: \text{or} \: \: (2, 2) \pmod{4}, \\ 
(d_{4} + d_{6} + d_{5} + d_{7}, d_{4} + d_{6} - d_{5} - d_{7}) &\equiv (0, 2) \: \: \text{or} \: \: (2, 0) \pmod{4} 
\end{align*}
and Lemma~$\ref{lem:2.5}$, we have 
\begin{align*}
\beta \overline{\beta} \in 
\begin{cases}
2^{5} \mathbb{Z}_{\rm odd}, &d \equiv d_{0} + d_{2} + d_{1} + d_{3} \equiv 2 \pmod{4}, \\ 
2^{6} \mathbb{Z}, & d \equiv d_{0} + d_{2} + d_{1} + d_{3} \equiv 0 \pmod{4}. 
\end{cases}
\end{align*}
Next, suppose that (ii) holds. 
Then, from Lemma~$\ref{lem:2.5}$, we have $\beta \overline{\beta} \in 2^{2} \mathbb{Z}_{\rm odd}$. 
Also, from 
\begin{align*}
(d_{0} - d_{2} - d_{5} + d_{7}, d_{0} - d_{2} + d_{5} - d_{7}) &\equiv (0, 0) \: \: \text{or} \: \: (2, 2) \pmod{4}, \\ 
(d_{4} - d_{6} + d_{1} - d_{3}, d_{4} - d_{6} - d_{1} + d_{3}) &\equiv (0, 2) \: \: \text{or} \: \: (2, 0) \pmod{4}
\end{align*}
and Lemma~$\ref{lem:2.5}$, we have 
\begin{align*}
\gamma \overline{\gamma} \in 
\begin{cases}
2^{5} \mathbb{Z}_{\rm odd}, & d \equiv d_{0} - d_{2} + d_{5} - d_{7} \equiv 2 \pmod{4}, \\ 
2^{6} \mathbb{Z}, & d \equiv d_{0} - d_{2} + d_{5} - d_{7} \equiv 0 \pmod{4}. 
\end{cases}
\end{align*}
\end{proof}

\begin{lem}\label{lem:5.4}
The following hold: 
\begin{enumerate}
\item[$(1)$] If $b_{0} \equiv b_{1} \equiv b_{2} \equiv b_{3} \equiv 1 \pmod{2}$ and $D_{4}(\bm{b}) D_{4}(\bm{c}) \in 2^{11} \mathbb{Z}_{\rm odd}$, then 
$$
b_{0} b_{2} + b_{1} b_{3} + c_{0} c_{2} + c_{1} c_{3} \equiv 2 \pmod{4}; 
$$
\item[$(2)$] If $b_{0} \equiv b_{2} \not\equiv b_{1} \equiv b_{3} \pmod{2}$, $D_{4}(\bm{b}) D_{4}(\bm{c}) \in 2^{11} \mathbb{Z}_{\rm odd}$ and $\beta \overline{\beta} \gamma \overline{\gamma} \in 2^{4} \mathbb{Z}_{\rm odd}$, then 
$$
b_{0} b_{2} + b_{1} b_{3} + c_{0} c_{2} + c_{1} c_{3} \equiv 2 \pmod{4}. 
$$
\end{enumerate}
\end{lem}
\begin{proof}
We prove (1). 
Let $b_{0} \equiv b_{1} \equiv b_{2} \equiv b_{3} \equiv 1 \pmod{2}$ and $D_{4}(\bm{b}) D_{4}(\bm{c}) \in 2^{11} \mathbb{Z}_{\rm odd}$. 
Then, from Remark~$\ref{rem:2.4}$~(1) and Lemma~$\ref{lem:3.2}$, 
either one of the following cases holds: 
\begin{enumerate}
\item[(i)] $D_{4}(\bm{b}) \in 2^{4} \mathbb{Z}_{\rm odd}$ and $D_{4}(\bm{c}) \in 2^{7} \mathbb{Z}_{\rm odd}$; 
\item[(ii)] $D_{4}(\bm{b}) \in 2^{7} \mathbb{Z}_{\rm odd}$ and $D_{4}(\bm{c}) \in 2^{4} \mathbb{Z}_{\rm odd}$. 
\end{enumerate}
If (i), then $b_{0} + b_{2} \not\equiv b_{1} + b_{3} \pmod{4}$ and $c_{0} + c_{2} \equiv c_{1} + c_{3} \equiv 0 \pmod{4}$ hold. 
Therefore, $(b_{0} b_{2} + b_{1} b_{3}, c_{0} c_{2} + c_{1} c_{3}) \equiv (0, 2) \pmod{4}$ holds. 
Thus, we have $b_{0} b_{2} + b_{1} b_{3} + c_{0} c_{2} + c_{1} c_{3} \equiv 2 \pmod{4}$. 
In the same way, the case~(ii) can also be proved. 
We prove (2). 
Let $b_{0} \equiv b_{2} \not\equiv b_{1} \equiv b_{3} \pmod{2}$, $D_{4}(\bm{b}) D_{4}(\bm{c}) \in 2^{11} \mathbb{Z}_{\rm odd}$ and $\beta \overline{\beta} \gamma \overline{\gamma} \in 2^{4} \mathbb{Z}_{\rm odd}$. 
Then, from Remarks~$\ref{rem:2.2}$ and $\ref{rem:2.4}$~(1), 
we may assume without loss of generality that $\bm{b} \equiv \bm{c} \equiv (0, 1, 0, 1) \pmod{2}$. 
From Lemma~$\ref{lem:3.2}$, 
either one of the following cases holds: 
\begin{enumerate}
\item[(iii)] $D_{4}(\bm{b}) \in 2^{5} \mathbb{Z}_{\rm odd}$ and $D_{4}(\bm{c}) \in 2^{6} \mathbb{Z}_{\rm odd}$; 
\item[(iv)] $D_{4}(\bm{b}) \in 2^{6} \mathbb{Z}_{\rm odd}$ and $D_{4}(\bm{c}) \in 2^{5} \mathbb{Z}_{\rm odd}$. 
\end{enumerate}
If (iii), then $b_{0} - b_{2} \equiv b_{1} - b_{3} \equiv 2 \pmod{4}$ and $c_{0} \equiv c_{2} \equiv 0 \: \: \text{or} \: \: 2 \pmod{4}$ hold. 
Therefore, $b_{0} b_{2} + b_{1} b_{3} \equiv - 1 \pmod{4}$ and $b_{0} + b_{1} + b_{2} + b_{3} \equiv 2 \pmod{4}$ hold. 
Also, since $b_{0} + b_{1} + b_{2} + b_{3} \not\equiv c_{0} + c_{1} + c_{2} + c_{3} \pmod{4}$ from Lemma~$\ref{lem:5.3}$, 
$c_{0} c_{2} + c_{1} c_{3} \equiv - 1 \pmod{4}$ holds. 
Thus, we have $b_{0} b_{2} + b_{1} b_{3} + c_{0} c_{2} + c_{1} c_{3} \equiv 2 \pmod{4}$. 
In the same way, the case~(iv) can also be proved. 
\end{proof}

\begin{lem}\label{lem:5.5}
If $b_{0} + b_{2} \equiv b_{1} + b_{3} \equiv 0 \pmod{2}$, $\beta \overline{\beta} \gamma \overline{\gamma} \in 2^{4} \mathbb{Z}_{\rm odd}$ and $b_{0} b_{2} + b_{1} b_{3} + c_{0} c_{2} + c_{1} c_{3} \equiv 2 \pmod{4}$, 
then $\beta \overline{\beta} \gamma \overline{\gamma} \in \left\{ 2^{4} p (2 m + 1) \mid p \in P, \: m \in \mathbb{Z} \right\}$. 
\end{lem}
\begin{proof}
Let $b_{0} + b_{2} \equiv b_{1} + b_{3} \equiv 0 \pmod{2}$ and $\beta \overline{\beta} \gamma \overline{\gamma} \in 2^{4} \mathbb{Z}_{\rm odd}$. 
From Remark~$\ref{rem:2.4}$~(1), 
we have 
$$
d_{0} + d_{2} + d_{4} + d_{6} \equiv d_{1} + d_{3} + d_{5} + d_{7} \equiv 0 \pmod{2}. 
$$
On the other hand, 
from Lemma~$\ref{lem:5.3}$, we have $b_{0} + b_{1} + b_{2} + b_{3} \not\equiv c_{0} + c_{1} + c_{2} + c_{3} \pmod{4}$. 
From this and Remark~$\ref{rem:2.4}$~(1) and (2), we have 
$$
2 d_{0} + 2 d_{2} \equiv b_{0} + b_{2} + c_{0} + c_{2} \not\equiv b_{1} + b_{3} + c_{1} + c_{3} \equiv 2 d_{1} + 2 d_{3} \pmod{4}. 
$$
Thus, $d_{0} + d_{2} \not\equiv d_{1} + d_{3} \pmod{2}$. 
From the above, we have $d_{0} + d_{2} \equiv d_{4} + d_{6} \not\equiv d_{1} + d_{3} \equiv d_{5} + d_{7} \pmod{2}$. 
Hence, 
\begin{align*}
&\left\{ (d_{0} + d_{2}) (d_{1} + d_{3}) + (d_{4} + d_{6}) (d_{5} + d_{7}) \right\}^{2} - \left\{ (d_{0} - d_{2}) (d_{5} - d_{7}) - (d_{4} - d_{6}) (d_{1} - d_{3}) \right\}^{2} \\ 
&\qquad \equiv 4 (d_{0} d_{2} + d_{4} d_{6} + d_{1} d_{3} + d_{5} d_{7}) \pmod{8}, \\ 
&(d_{0}^{2} + d_{2}^{2} + d_{4}^{2} + d_{6}^{2} + d_{1}^{2} + d_{3}^{2} + d_{5}^{2} + d_{7}^{2}) (d_{0} d_{2} + d_{4} d_{6} + d_{1} d_{3} + d_{5} d_{7}) \equiv 0 \pmod{4}. 
\end{align*}
Therefore, from Lemmas~$\ref{lem:2.8}$, $\ref{lem:2.9}$ and $\ref{lem:2.10}$~(1), 
\begin{align*}
\beta \overline{\beta} - \gamma \overline{\gamma} 
&\equiv 8 (d_{0}^{2} + d_{2}^{2} + d_{4}^{2} + d_{6}^{2} + d_{1}^{2} + d_{3}^{2} + d_{5}^{2} + d_{7}^{2}) (d_{0} d_{2} + d_{4} d_{6} + d_{1} d_{3} + d_{5} d_{7}) \\ 
&\qquad + 16 (d_{0} d_{2} + d_{4} d_{6} + d_{1} d_{3} + d_{5} d_{7}) \\ 
&\equiv 16 (d_{0} d_{2} + d_{4} d_{6} + d_{1} d_{3} + d_{5} d_{7}) \\ 
&\equiv 8 (b_{0} b_{2} + b_{1} b_{3} + c_{0} c_{2} + c_{1} c_{3}) \pmod{32}. 
\end{align*}
Let $b_{0} b_{2} + b_{1} b_{3} + c_{0} c_{2} + c_{1} c_{3} \equiv 2 \pmod{4}$. 
Then we have $\beta \overline{\beta} - \gamma \overline{\gamma} \equiv 16 \pmod{32}$. 
Note that $\beta \overline{\beta} \equiv 4$ or $- 12 \pmod{32}$ since $\beta \overline{\beta} \equiv 4 \pmod{16}$ from Lemma~$\ref{lem:2.8}$. 
If $\beta \overline{\beta} \equiv 4 \pmod{32}$, then $\gamma \overline{\gamma} \equiv 4 - 16 \equiv - 12 \pmod{32}$. 
This implies that $\gamma \overline{\gamma}$ has at least one prime factor of the form $8 k - 3$. 
If $\beta \overline{\beta} \equiv - 12 \pmod{32}$, then $\beta \overline{\beta}$ has at least one prime factor of the form $8 k - 3$. 
\end{proof}

\begin{lem}\label{lem:5.6}
Let $b_{0} + b_{2} \equiv b_{1} + b_{3} \equiv 1 \pmod{2}$, 
$D_{4}(\bm{b}) D_{4}(\bm{c}) \in 2^{8} \mathbb{Z}_{{\rm odd}}$ and $\beta \overline{\beta} \gamma \overline{\gamma} \in 2^{7} \mathbb{Z}_{{\rm odd}}$. 
Then $D_{4}(\bm{b}) D_{4}(\bm{c}) \beta \overline{\beta} \gamma \overline{\gamma} \in \left\{ 2^{15} p (2 m + 1) \mid p \in P, \: m \in \mathbb{Z} \right\}$. 
\end{lem}
\begin{proof}
From Lemmas~$\ref{lem:5.2}$ and $\ref{lem:5.3}$, 
we have 
$$
(b_{0} + b_{2}) (b_{1} + b_{3}) \equiv \pm 3, \quad (c_{0} + c_{2}) (c_{1} + c_{3}) \equiv \pm 3 \pmod{8}, \quad d \equiv 2 \pmod{4}, 
$$
where 
$$
d := \left\{ (d_{0} + d_{2}) (d_{5} + d_{7}) + (d_{4} + d_{6}) (d_{1} + d_{3}) \right\} \left\{ (d_{0} - d_{2}) (d_{1} - d_{3}) + (d_{4} - d_{6}) (d_{5} - d_{7}) \right\}. 
$$
We divide the proof into the following cases: 
\begin{enumerate}
\item[(i)] $(b_{0} b_{3} + b_{2} b_{1}, c_{0} c_{3} + c_{2} c_{1}) \equiv (0, 0), \: \: (0, 2) \: \: \text{or} \: \: (2, 0) \pmod{4}$; 
\item[(ii)] $(b_{0} b_{3} + b_{2} b_{1}, c_{0} c_{3} + c_{2} c_{1}) \equiv (2, 2) \pmod{4}$; 
\item[(iii)] $(b_{0} b_{1} + b_{2} b_{3}, c_{0} c_{1} + c_{2} c_{3}) \equiv (0, 0), \: \: (0, 2) \: \: \text{or} \: \: (2, 0) \pmod{4}$; 
\item[(iv)] $(b_{0} b_{1} + b_{2} b_{3}, c_{0} c_{1} + c_{2} c_{3}) \equiv (2, 2) \pmod{4}$. 
\end{enumerate}
First, we consider the case (i). 
If $b_{0} b_{3} + b_{2} b_{1} \equiv 0 \pmod{4}$, 
then 
$$
(b_{0} - b_{2}) (b_{1} - b_{3}) = (b_{0} + b_{2}) (b_{1} + b_{3}) - 2 (b_{0} b_{3} + b_{2} b_{1}) \equiv \pm 3 \pmod{8}. 
$$
Thus, $(b_{0} - b_{2})^{2} + (b_{1} - b_{3})^{2} \equiv - 6 \pmod{16}$. 
It implies that $(b_{0} - b_{2})^{2} + (b_{1} - b_{3})^{2}$ has at least one prime factor of the form $8 k - 3$. 
That is, $D_{4}(\bm{b})$ has at least one prime factor of the form $8 k - 3$. 
In the same way, we can prove that $D_{4}(\bm{c})$ has at least one prime factor of the form $8 k - 3$ when $c_{0} c_{3} + c_{2} c_{1} \equiv 0 \pmod{4}$. 
Therefore, if (i), then 
$$
D_{4}(\bm{b}) D_{4}(\bm{c}) \in \left\{ 2^{8} p (2 m + 1) \mid p \in P, \: m \in \mathbb{Z} \right\}. 
$$
We can obtain the same conclusion for the case~(iii). 
Next, we consider the case~(ii). 
From $b_{0} + b_{2} \equiv b_{1} + b_{3} \equiv 1 \pmod{2}$ and Remark~$\ref{rem:2.4}$~(1), 
we have $d_{0} + d_{2} + d_{4} + d_{6} \equiv d_{1} + d_{3} + d_{5} + d_{7} \equiv 1 \pmod{2}$. 
From Remark~$\ref{rem:2.6}$, we may assume without loss of generality that $d_{0} + d_{2} \equiv 0$, $d_{4} + d_{6} \equiv 1 \pmod{2}$. 
Then, either one of the following cases holds: 
\begin{enumerate}
\item[(ii-1)] $d_{0} + d_{2} \equiv d_{1} + d_{3} \equiv 0$, \: $d_{4} + d_{6} \equiv d_{5} + d_{7} \equiv 1 \pmod{2}$; 
\item[(ii-2)] $d_{0} + d_{2} \equiv d_{5} + d_{7} \equiv 0$, \: $d_{4} + d_{6} \equiv d_{1} + d_{3} \equiv 1 \pmod{2}$. 
\end{enumerate}
Suppose that (ii) and (ii-1) hold. 
Then, from Lemma~$\ref{lem:2.10}$~(2), 
it follows that 
\begin{align*}
(d_{0} - d_{2}) + (d_{1} - d_{3}) 
&\equiv (d_{0} - d_{2}) (d_{5} - d_{7}) + (d_{4} - d_{6}) (d_{1} - d_{3}) \\ 
&\equiv (d_{0} + d_{2}) (d_{5} + d_{7}) + (d_{4} + d_{6}) (d_{1} + d_{3}) \\ 
&\qquad + 2 (d_{0} d_{7} + d_{2} d_{5} + d_{4} d_{3} + d_{6} d_{1}) \\ 
&\equiv d + (b_{0} b_{3} + b_{2} b_{1}) - (c_{0} c_{3} + c_{2} c_{1}) \\ 
&\equiv 2 \pmod{4}. 
\end{align*}
Therefore, from Lemma~$\ref{lem:2.8}$, we have 
$\gamma \overline{\gamma} \equiv 4 - 16 \equiv - 12 \pmod{32}$. 
Thus, $\gamma \overline{\gamma}$ has at least one prime factor of the form $8 k - 3$. 
Suppose that (ii) and (ii-2) hold. 
Then, from Lemma~$\ref{lem:2.10}$~(3), 
it follows that 
\begin{align*}
(d_{0} + d_{2}) + (d_{5} + d_{7}) 
&\equiv (d_{0} + d_{2}) (d_{1} + d_{3}) + (d_{4} + d_{6}) (d_{5} + d_{7}) \\ 
&\equiv (d_{0} - d_{2}) (d_{1} - d_{3}) + (d_{4} - d_{6}) (d_{5} - d_{7}) \\ 
&\qquad + 2 (d_{0} d_{3} + d_{2} d_{1} + d_{4} d_{7} + d_{6} d_{5}) \\ 
&\equiv d + (b_{0} b_{3} + b_{2} b_{1}) + (c_{0} c_{3} + c_{2} c_{1}) \\ 
&\equiv 2 \pmod{4}. 
\end{align*}
Therefore, from Lemma~$\ref{lem:2.8}$, we have 
$\beta \overline{\beta} \equiv 4 - 16 \equiv - 12 \pmod{32}$. 
Thus, $\beta \overline{\beta}$ has at least one prime factor of the form $8 k - 3$. 
From the above, if (ii), then 
$$
\beta \overline{\beta} \gamma \overline{\gamma} \in \left\{ 2^{7} p (2 m + 1) \mid p \in P, \: m \in \mathbb{Z} \right\}. 
$$
We can obtain the same conclusion for the case~(iv) by using Lemma~$\ref{lem:2.10}$~(4) and (5). 
\end{proof}

\begin{proof}[Proof of Lemma~$\ref{lem:5.1}$]
We prove (1). 
Let $D_{4 \times 4}(\bm{a}) = D_{4}(\bm{b}) D_{4}(\bm{c}) \beta \overline{\beta} \gamma \overline{\gamma} \in 2 \mathbb{Z}$. 
Then, from Lemma~$\ref{lem:2.7}$, we have $D_{4}(\bm{b}) \in 2 \mathbb{Z}$. 
Thus, $b_{0} + b_{2} \equiv b_{1} + b_{3} \pmod{2}$ from Corollary~$\ref{cor:2.1}$. 
Therefore, from Lemmas~$\ref{lem:5.2}$ and $\ref{lem:5.3}$, 
we obtain (1). 
We prove (2). 
Let $D_{4 \times 4}(\bm{a}) \in 2^{15} \mathbb{Z}_{\rm odd}$. 
Then, from Lemmas~$\ref{lem:5.2}$ and $\ref{lem:5.3}$, 
one of the following cases holds: 
\begin{enumerate}
\item[(i)] $b_{0} \equiv b_{1} \equiv b_{2} \equiv b_{3} \equiv 1 \pmod{2}$, $D_{4}(\bm{b}) D_{4}(\bm{c}) \in 2^{11} \mathbb{Z}_{\rm odd}$ and $\beta \overline{\beta} \gamma \overline{\gamma} \in 2^{4} \mathbb{Z}_{\rm odd}$; 
\item[(ii)] $b_{0} \equiv b_{2} \not\equiv b_{1} \equiv b_{3} \pmod{2}$, $D_{4}(\bm{b}) D_{4}(\bm{c}) \in 2^{11} \mathbb{Z}_{\rm odd}$ and $\beta \overline{\beta} \gamma \overline{\gamma} \in 2^{4} \mathbb{Z}_{\rm odd}$; 
\item[(iii)] $b_{0} + b_{2} \equiv b_{1} + b_{3} \equiv 1 \pmod{2}$, $D_{4}(\bm{b}) D_{4}(\bm{c}) \in 2^{8} \mathbb{Z}_{\rm odd}$ and $\beta \overline{\beta} \gamma \overline{\gamma} \in 2^{7} \mathbb{Z}_{\rm odd}$. 
\end{enumerate}
Therefore, from Lemmas~$\ref{lem:5.4}$--$\ref{lem:5.6}$, 
we obtain (2). 
\end{proof}

\section{Possible numbers}
In this section, 
we determine all possible numbers. 
Lemmas~$\ref{lem:4.1}$ and $\ref{lem:5.1}$ imply that $S \left( {\rm C}_{4}^{2} \right)$ does not include every integer that is not mentioned in Lemmas~$\ref{lem:6.1}$--$\ref{lem:6.3}$. 

\begin{lem}\label{lem:6.1}
For any $m \in \mathbb{Z}$, 
the following are elements of $S({\rm C}_{4}^{2})$: 
\begin{enumerate}
\item[$(1)$] $16 m + 1$; 
\item[$(2)$] $2^{16} (4 m + 1)$; 
\item[$(3)$] $2^{16} (4 m - 1)$; 
\item[$(4)$] $2^{16} (2 m)$. 
\end{enumerate}
\end{lem}

\begin{lem}\label{lem:6.2}
For any $k \in \mathbb{Z}$, $16 l - 3$, $16 m - 3$, $16 n - 3$, 
$16 l + 5$, $16 m + 5$, $16 n + 5 \in P$, 
the following are elements of $S({\rm C}_{4}^{2})$: 
\begin{enumerate}
\item[$(1)$] $(16 k - 3) (16 l + 5) (16 m - 3) (16 n - 3)$; 
\item[$(2)$] $(16 k - 3) (16 l + 5) (16 m + 5) (16 n + 5)$; 
\item[$(3)$] $(16 k + 5) (16 l - 3) (16 m - 3) (16 n - 3)$; 
\item[$(4)$] $(16 k + 5) (16 l - 3) (16 m + 5) (16 n + 5)$. 
\end{enumerate}
\end{lem}

\begin{lem}\label{lem:6.3}
For any $m \in \mathbb{Z}$ and $p \in P$, 
we have $2^{15} p (2 m + 1) \in S({\rm C}_{4}^{2})$. 
\end{lem}

\begin{proof}[Proof of Lemma~$\ref{lem:6.1}$]
We obtain (1) from 
\begin{align*}
D_{4 \times 4}(m + 1, m, \ldots, m) 
&= D_{4}(4 m + 1, 4 m, 4 m, 4 m) D_{4}(1, 0, 0, 0)^{3} \\ 
&= (8 m + 1)^{2} - (8 m)^{2} \\ 
&= 16 m + 1. 
\end{align*}
We obtain (2) from 
\begin{align*}
&D_{4 \times 4}(m + 2, m, m, m, m, m, m + 1, m, m + 1, m, \ldots, m) \\ 
&\qquad = D_{4}(4 m + 3, 4 m, 4 m + 1, 4 m) D_{4}(3, 0, - 1, 0) D_{4}\left(1, 0, \sqrt{- 1}, 0 \right) D_{4}\left(1, 0, - \sqrt{- 1}, 0 \right) \\ 
&\qquad = \left\{ (8 m + 4)^{2} - (8 m)^{2} \right\} \cdot 2^{2} \cdot 2^{2} \cdot 4^{2} \cdot \left( 1 + \sqrt{- 1} \right)^{4} \left( 1 - \sqrt{- 1} \right)^{4} \\ 
&\qquad = 2^{16} (4 m + 1). 
\end{align*}
We obtain (3) from 
\begin{align*}
&D_{4 \times 4}(m + 1, m, m, m - 1, m, m - 1, m, m, m, m, m, m - 1, m, m - 1, m - 1, m) \\ 
&\quad = D_{4}(4 m + 1, 4 m - 2, 4 m - 1, 4 m - 2) D_{4}(1, 2, 1, - 2) D_{4}\left( 1, 0, \sqrt{- 1}, 0 \right) D_{4}\left( 1, 0, - \sqrt{- 1}, 0 \right) \\ 
&\quad = \left\{ (8 m)^{2} - (8 m - 4)^{2} \right\} \cdot 2^{2} \cdot 2^{2} \cdot 4^{2} \cdot \left( 1 + \sqrt{- 1} \right)^{4} \left( 1 - \sqrt{- 1} \right)^{4} \\ 
&\quad = 2^{16} (4 m - 1). 
\end{align*}
We obtain (4) from 
\begin{align*}
&D_{4 \times 4}(m + 1, m, m, m, m, m, m, m, m + 1, m - 1, m, m, m, m - 1, m, m) \\ 
&\qquad = D_{4}(4 m + 2, 4 m - 2, 4 m, 4 m) D_{4}(2, 0, 0, 0) D_{4}\left( 0, 1 + \sqrt{- 1}, 0, 0 \right) D_{4}(0, 1 - \sqrt{- 1}, 0, 0) \\ 
&\qquad = \left\{ (8 m + 2)^{2} - (8 m - 2)^{2} \right\} \cdot 2^{3} \cdot 2^{2} \cdot 2^{2} \cdot (1 + \sqrt{- 1})^{4} (1 - \sqrt{- 1})^{4} \\ 
&\qquad = 2^{16} (2 m). 
\end{align*}
\end{proof}

\begin{rem}\label{rem:6.4}
When 
\begin{align*}
d_{0} &= 2 t - 2 v, & 
d_{1} &= 2 t + 2 w + 1, & 
d_{2} &= 2 t + 2 v + 2 e, & 
d_{3} &= 2 t - 2 w, \\ 
d_{4} &= 2 u + 2 w + 1, & 
d_{5} &= 2 u + 2 v + 1, & 
d_{6} &= 2 u - 2 w, & 
d_{7} &= 2 u - 2 v, 
\end{align*}
where $e \in \{ 0, 1 \}$, 
from Lemma~$\ref{lem:2.5}$, 
we have 
$$
\beta \overline{\beta} \gamma \overline{\gamma} = \left\{ (8 t + 2 e + 1)^{2} + (8 u + 2)^{2} \right\} \left\{ (8 v + 2 e + 1)^{2} + (8 w + 2)^{2} \right\}. 
$$
\end{rem}

\begin{proof}[Proof of Lemma~$\ref{lem:6.2}$]
We remark that for any $16 m - 8 e + 5, 16 n - 8 e + 5 \in P$ with $e \in \{ 0, 1 \}$, 
there exist $t, u, v, w \in \mathbb{Z}$ satisfying 
\begin{equation}\tag{$\ast$}
\begin{split}
16 m - 8 e + 5 &= (8 t + 2 e + 1)^{2} + (8 u + 2)^{2}, \\ 
16 n - 8 e + 5 &= (8 v + 2 e + 1)^{2} + (8 w + 2)^{2}. 
\end{split}
\end{equation}
We prove (1) and (2). 
For any $16 l + 5 \in P$, 
we can take $r, s \in \mathbb{Z}$ satisfying 
\begin{align*}
16 l + 5 &= (8 r + 1)^{2} + (8 s + 2)^{2}. 
\end{align*}
Also, we take $t, u, v, w$ satisfying ($\ast$). 
Let 
\begin{align*}
a_{0} &= k - r + t - v, &
a_{1} &= k - s + t + w, &
a_{2} &= k + r + t + v + e, \\ 
a_{3} &= k + s + t - w, &
a_{4} &= - k + r + u + w + 1, &
a_{5} &= - k + s + u + v + 1, \\ 
a_{6} &= - k - r + u - w, &
a_{7} &= - k - s + u - v, &
a_{8} &= k - r - t + v, \\ 
a_{9} &= k - s - t - w - 1, &
a_{10} &= k + r - t - v - e, &
a_{11} &= k + s - t + w, \\ 
a_{12} &= - k + r - u - w, &
a_{13} &= - k + s - u - v, &
a_{14} &= - k - r - u + w, \\ 
a_{15} &= - k - s - u + v. 
\end{align*}
Then, from Remark~$\ref{rem:6.4}$, we have 
\begin{align*}
D_{4 \times 4}(\bm{a}) &= D_{4}(1, 0, 0, 0) D_{4}(4 k - 4 r - 1, 4 k - 4 s - 2, 4 k + 4 r, 4 k + 4 s) \beta \overline{\beta} \gamma \overline{\gamma} \\ 
&= \left\{ (8 k - 1)^{2} - (8 k - 2)^{2} \right\} \left\{ (8 r + 1)^{2} + (8 s + 2)^{2} \right\} \\ 
&\qquad \times \left\{ (8 t + 2 e + 1)^{2} + (8 u + 2)^{2} \right\} \left\{ (8 v + 2 e + 1)^{2} + (8 w + 2)^{2} \right\} \\ 
&= (16 k - 3) (16 l + 5) (16 m - 8 e + 5) (16 n - 8 e + 5). 
\end{align*}
We prove (3) and (4). 
For any $16 l - 3 \in P$, 
we can take $r, s \in \mathbb{Z}$ satisfying 
\begin{align*}
16 l - 3 &= (8 r + 3)^{2} + (8 s + 2)^{2}. 
\end{align*}
Also, we take $t, u, v, w$ satisfying ($\ast$). 
Let 
\begin{align*}
a_{0} &= k + r + t - v + 1, &
a_{1} &= k + s + t + w + 1, &
a_{2} &= k - r + t + v + e, \\ 
a_{3} &= k - s + t - w, &
a_{4} &= - k - r + u + w, &
a_{5} &= - k - s + u + v, \\ 
a_{6} &= - k + r + u - w, &
a_{7} &= - k + s + u - v, &
a_{8} &= k + r - t + v + 1, \\ 
a_{9} &= k + s - t - w, &
a_{10} &= k - r - t - v - e, &
a_{11} &= k - s - t + w, \\ 
a_{12} &= - k - r - u - w - 1, &
a_{13} &= - k - s - u - v - 1, &
a_{14} &= - k + r - u + w, \\ 
a_{15} &= - k + s - u + v. 
\end{align*}
Then, from Remark~$\ref{rem:6.4}$, we have 
\begin{align*}
D_{4 \times 4}(\bm{a}) &= D_{4}(1, 0, 0, 0) D_{4}(4 k + 4 r + 3, 4 k + 4 s + 2, 4 k - 4 r, 4 k - 4 s) \beta \overline{\beta} \gamma \overline{\gamma} \\ 
&= \left\{ (8 k + 3)^{2} - (8 k + 2)^{2} \right\} \left\{ (8 r + 3)^{2} + (8 s + 2)^{2} \right\} \\ 
&\: \: \qquad \times \left\{ (8 t + 2 e + 1)^{2} + (8 u + 2)^{2} \right\} \left\{ (8 v + 2 e + 1)^{2} + (8 w + 2)^{2} \right\} \\ 
&= (16 k + 5) (16 l - 3) (16 m - 8 e + 5) (16 n - 8 e + 5). 
\end{align*}
\end{proof}

\begin{proof}[Proof of Lemma~$\ref{lem:6.3}$]
For any $p \in P$, 
there exist $r, s \in \mathbb{Z}$ satisfying 
$2 p = (8 r + 3)^{2} + (8 s + 1)^{2}$. 
Let 
\begin{align*}
a_{0} &= m + r + 1, &
a_{1} &= m + r + 1, &
a_{2} &= m + r + 1, &
a_{3} &= m + r, \\ 
a_{4} &= m + s, & 
a_{5} &= m + s, &
a_{6} &= m + s + 1, &
a_{7} &= m + s, \\ 
a_{8} &= m - r + 1, & 
a_{9} &= m - r, &
a_{10} &= m - r, &
a_{11} &= m - r - 1, \\ 
a_{12} &= m - s, &
a_{13} &= m - s, &
a_{14} &= m - s, &
a_{15} &= m - s. 
\end{align*}
Then, from Lemma~$\ref{lem:2.5}$, 
we have $\beta \overline{\beta} \gamma \overline{\gamma} = 2^{3} \left\{ (8 r + 3)^{2} + (8 s + 1)^{2} \right\} = 2^{4} p$. 
Hence, we have 
\begin{align*}
D_{4 \times 4}(\bm{a}) 
&= D_{4}(4 m + 2, 4 m + 1, 4 m + 2, 4 m - 1) D_{4}(2, 1, 0, - 1) \beta \overline{\beta} \gamma \overline{\gamma} \\ 
&= \left\{ (8 m + 4)^{2} - (8 m)^{2} \right\} \cdot 2^{2} \cdot 2^{2} \cdot (2^{2} + 2^{2}) \cdot 2^{4} p \\ 
&= 2^{15} p (4 m + 1). 
\end{align*}
On the other hand, 
let 
\begin{align*}
a_{0} &= m + r, &
a_{1} &= m + r, &
a_{2} &= m + r + 1, &
a_{3} &= m + r, \\ 
a_{4} &= m + s, & 
a_{5} &= m + s, &
a_{6} &= m + s, &
a_{7} &= m + s - 1, \\ 
a_{8} &= m - r, & 
a_{9} &= m - r - 1, &
a_{10} &= m - r, &
a_{11} &= m - r - 1, \\ 
a_{12} &= m - s, &
a_{13} &= m - s, &
a_{14} &= m - s - 1, &
a_{15} &= m - s - 1. 
\end{align*}
Then, from Lemma~$\ref{lem:2.5}$, 
we have $\beta \overline{\beta} \gamma \overline{\gamma} = 2^{3} \left\{ (8 r + 3)^{2} + (8 s + 1)^{2} \right\} = 2^{4} p$. 
Hence, we have 
\begin{align*}
D_{4 \times 4}(\bm{a}) 
&= D_{4}(4 m, 4 m - 1, 4 m, 4 m - 3) D_{4}(0, - 1, 2, 1) \beta \overline{\beta} \gamma \overline{\gamma} \\ 
&= \left\{ (8 m)^{2} - (8 m - 4)^{2} \right\} \cdot 2^{2} \cdot 2^{2} \cdot (2^{2} + 2^{2}) \cdot 2^{4} p \\ 
&= 2^{15} p (4 m - 1). 
\end{align*}
\end{proof}

From Lemmas~$\ref{lem:4.1}$, $\ref{lem:5.1}$ and $\ref{lem:6.1}$--$\ref{lem:6.3}$, 
Theorem~$\ref{thm:1.1}$ is proved.

\clearpage

\bibliography{reference}
\bibliographystyle{plain}

\medskip
\begin{flushleft}
Faculty of Education, 
University of Miyazaki, 
1-1 Gakuen Kibanadai-nishi, 
Miyazaki 889-2192, 
Japan \\ 
{\it Email address}, Yuka Yamaguchi: y-yamaguchi@cc.miyazaki-u.ac.jp \\ 
{\it Email address}, Naoya Yamaguchi: n-yamaguchi@cc.miyazaki-u.ac.jp 
\end{flushleft}

\end{document}